\documentclass[11pt,reqno]{amsart}
\usepackage{amssymb}
\usepackage{mathptmx}
\usepackage{mathtools}
\usepackage{mathrsfs}
\usepackage[export]{adjustbox}
\usepackage[all]{xy}
\usepackage{stmaryrd}
\usepackage{fancyhdr}
\usepackage{hyperref}
\usepackage{tikz-cd}
\usepackage{cite}
\usepackage{amsmath,amsfonts}
\usepackage{graphicx}
\usepackage{wrapfig}
\usepackage{array}
\usepackage{amsthm}
\usepackage[left=1.2in, right=1.2in, top=1in, bottom=1in]{geometry}
\usepackage{etoolbox}
\patchcmd{\section}{\scshape}{\bfseries}{}{}
\makeatletter
\renewcommand{\@secnumfont}{\bfseries}
\makeatother
\xyoption{all}
\usepackage{secdot}

\theoremstyle{plain}
\DeclareMathOperator{\id}{\textrm{id}}

\input xy
\xyoption{all}

\theoremstyle{definition}
\newtheorem{mydef}{\textbf{Definition}}[section]
\newtheorem{myeg}[mydef]{\textbf{Example}}

\newtheorem{mythm}[mydef]{\textbf{Theorem}}

\newtheorem{rmk}[mydef]{\textbf{Remark}}
\theoremstyle{plain}

\newtheorem*{nothma}{\textbf{Theorem A}}
\newtheorem*{nothmb}{\textbf{Theorem B}}
\newtheorem*{nothmc}{\textbf{Theorem C}}

\newtheorem{lem}[mydef]{\textbf{Lemma}}
\newtheorem{pro}[mydef]{\textbf{Proposition}}

\newtheorem{cor}[mydef]{\textbf{Corollary}}

\newtheorem*{notation}{\textbf{Notation}}

\tikzset{main node/.style={circle,fill=black,draw,minimum size=0.3cm,inner sep=0pt},
}

\begin{document}

	\title{On the Hopf algebra of multi-complexes}
	
	\author{Miodrag Iovanov}
	\address{University of Iowa, IA, USA and Simion Stoilow Institute of the Romanian Academy, Bucharest, Romania}
	\curraddr{}
	\email{miodrag-iovanov@uiowa.edu}
	
	\author{Jaiung Jun}
	\address{State University of New York at New Paltz, NY, USA}
	\curraddr{}
	\email{junj@newpaltz.edu}

	
	\dedicatory{}

	\maketitle
	
\begin{abstract}
We introduce a general class of combinatorial objects, which we call \emph{multi-complexes}, which simultaneously generalizes graphs, multigraphs, hypergraphs and simplicial and delta complexes. We introduce a natural algebra of multi-complexes which is defined as the algebra which has a formal basis $\mathcal{C}$ of all isomorphism types of multi-complexes, and multiplication is to take the disjoint union. This is a Hopf algebra with an operation encoding the dissasembly information for such objects, and extends the  Hopf algebra of graphs. In our main result, we explicitly describe here the structure of this Hopf algebra of multi-complexes $H$. We find an explicit basis $\mathcal{B}$ of the space of primitives, which is of combinatorial relevance: it is such that each multi-complex is a polynomial with non-negative integer coefficients of the elements of $\mathcal{B}$, and each $b\in\mathcal{B}$ is a polynomial with integer coefficients in $\mathcal{C}$. Using this, we find the cancellation and grouping free formula for the antipode. The coefficients appearing in all these polynomials are, up to signs, numbers counting multiplicities of sub-multi-complexes in a multi-complex. We also explicitly illustrate how our results specialize to the graph Hopf algebra, and observe how they specialize to results in all of the above mentioned particular cases. We also investigate applications of these results to the graph reconstruction conjectures, and rederive some results in the literature on these questions.  
\end{abstract}
	
	\section{Introduction}

	Hopf algebras are now ubiquitous in many fields of mathematics. In combinatorics, Hopf algebras appear naturally when studying various classes of combinatorial objects, such as graphs, matroids, posets or symmetric functions. Given such a class of objects, one associates an algebra which captures combinatorial properties and constructions of the underlying class; the generators of the algebra are the isomorphism types of such objects, and basic information on these objects regarding assembly and disassembly operations, are encoded in the algebraic structure of a Hopf algebra. 
	In this respect, the product of two objects is obtained by a natural combinatorial construction which puts them together and assembles them into a new object (such as taking disjoint union), and coproduct encodes all possible ways to ``split'' the given combinatorial object into two parts in a suitable way\footnote{Strictly speaking, combinatorial Hopf algebras sometimes mean Hopf algebras satisfying certain conditions (for instance, see \cite{loday2010combinatorial} or \cite{aguiar2006combinatorial}).}. 
	Finally, one hopes to use certain algebraic properties of a Hopf algebra to return to combinatorics and obtain new information, such as new combinatorial identities.  For instance, in\cite{crapo2005free}, Crapo and Schmitt used this line of thought to give a short proof of Welsh’s conjecture on a lower bound for the number of isomorphism classes of matroids on $\{1,...,n\}$. For another example, in \cite{eppolito2018proto}, Eppolito, Jun, and Szczesny provided another interpretation of the combinatorial Hopf algebra arising from matroids as the Hall algebra associated to the category of matroids. For an introduction to Hopf algebras in combinatorics, we refer the reader to \cite{grinberg2014hopf}.
	
	When one constructs Hopf algebras by the aforementioned recipe, one usually has only explicit description of product and coproduct, which usually yields a connected, graded bialgebra. In this case, by classical Hopf algebra theory (see, for example, \cite{takeuchi1971free}), one knows that such a bialgebra has an antipode and hence is a Hopf algebra. 
	
	One of the main questions regarding any such combinatorial algebra, and which has received a lot of attention, is to explicitly find the antipode. However, the antipode formulas usually involve massive cancellations and re-groupings of terms, and hence in many cases such an antipode formula is not optimal for use in relation to certain combinatorial identities. 
	
	Recently, considerable attention has been dedicated to finding cancellation-free formulas for the antipode of various Hopf algebras. In their groundbreaking work \cite{aguiar2017hopf}, Aguiar and Ardila provided an elegant unified way to find a cancellation-free and grouping-free antipode formula for various classes of combinatorial Hopf algebras by reducing the question to the case of generalized permutahedra (or polymatroids). In \cite{bucher2018matroid}, Bucher, Eppolito, Jun, and Matherne also employed the idea of sign-reversing involution, which was introduced in \cite{benedetti2017antipodes} by Benedetti and Sagan, and provided a cancellation-free antipode formula for the matroid-minor Hopf algebra. This approach can be also used to provide a cancellation-free antipode formula for the Hopf algebras defined by Eppolito, Jun, and Szczesny in \cite{eppolito2017hopf}, arising from matroids over hyperfields as in Baker-Bowler \cite{baker2017matroids}.
	
	In contrast to the case of the antipode formulas,  primitive elements of combinatorial Hopf algebras seem to have received relatively less attention. In many interesting cases, combinatorial Hopf algebras are connected, commutative, and cocommutative. Hence, by the classical Milnor–Moore theorem, they are isomorphic to a polynomial algebra as Hopf algebras; more precisely, if $B$ is a linear basis in the space of primitives of such a Hopf algebra $H$, then $H$ is the polynomial algebra in the elements of $B$ (isomorphic to the polynomial algebra in $|B|$ variables). While the original set of combinatorial objects used to build $H$ is also often an algebraic monomial basis, those elements are not primitives, and this identification with a polynomial algebra is not a Hopf algebra isomorphism.
	
	However, the problem of determining a good basis in the space of primitives is of central importance: if one knows explicitly a basis of the space of primitive elements, one can translate combinatorial questions to questions concerning polynomial algebras in an explicit way. For instance, in \cite{schmitt1995hopf}, Schmitt studied invariants of combinatorial objects in this way. To this end, it is beneficial to know an explicit description of a basis of the space of primitive elements of a combinatorial Hopf algebra, which provides a natural (Hopf) isomorphism between combinatorial Hopf algebras and polynomial algebras. In the case of the matroid-minor Hopf algebra, Crapo and Schmitt found explicitly two bases of the space of primitive elements in \cite{crapo2008primitive} by introducing a new operation (free product) of matroids.\footnote{Note that the matroid-minor Hopf algebra is not cocommutative.}
	
	Hence, one can pose the following general problem, which asks to describe the Hopf structure of a combinatorial algebra in a meaningful way which relates to combinatorics:
	
	\vspace{.3cm}
	
	\noindent {\bf Problem.} \, {\it Given some combinatorial Hopf algebra $H$, constructed from a certain class of combinatorial objects and their assembly-disassembly operations, find the precise Hopf algebra structure of $H$ by giving: \\
		(1) a ``good" basis of the space of primitives, which relates to the original combinatorial basis in such a way that the coefficients of the algebraic base change have combinatorial significance;\\
		(2) a cancellation and grouping-free formula for the antipode in terms of the combinatorial basis.}
	\vspace{.3cm}
	
The main idea of this paper is to introduce a new very general class of combinatorial objects, which we call \emph{multi-complexes}, which include and generalize many classes of other important objects, such as graphs, hypergraphs, multigraphs, simplicial complexes, colored complexes, etc. There is an associated Hopf algebra $H_{\mathcal C}$ of isomorphism types of multi-complexes, constructed by the above recipe. In our main results, we solve the above problems for this general algebra, finding bases of combinatorial relevance for the set of primitives, as well as  grouping and multiplicity-free formulas for the antipode. As multi-complexes generalize many other structures, in particular, this result specializes and provides a solution for the above problem for all these types of combinatorial objects, such as graphs, hypergraphs and simplicial complexes, among others. 
	
Roughly speaking, a multi-complex $C$ is a finite multiset $\{A_i\}_i$ of multisets based on a finite set $n_C$ together with a partial order $\preceq$ satisfying certain conditions (see, Definition \ref{definition: multi-complex}.); intuitively, $n_C$ is ``the set of vertices'', $\{A_i\}$ is ``the set of faces'', and $\preceq$ encodes ``the gluing information'' for faces. 

The Hopf algebra of multi-complexes is constructed following the typical recipe (see, Lemma \ref{lemma: multi-complex Hopf} and the paragraph before it): it has as a formal basis (over some field) the set $\mathcal {C}$ of isomorphism types of multi-complexes. The product of $H_{\mathcal C}$ is a (suitably defined) disjoint union of multi-complexes, and coproduct of a multi-complex $C$ is obtained by splitting $C$ in all possible ways into two ``induced sub-multi-complexes (see \S \ref{Section: Hopf algebra of multi-complexes} for the precise definition and formula). The set of isomorphism types of ``connected'' multi-complexes provides a polynomial basis of this algebra. 
	
Let $C$ be a multi-complex based on a finite set $n_C$. If $D$ is a sub-multi-complex of $C$ which is also based on $n_C$, then we denote $D \preceq C$. The key observation is the consideration of the following elements in $H_\mathcal{C}$, which we introduce here:
	\[
	P_C:= \sum_{D \preceq C} \mu_P(D,C)D, 
	\]
where $\mu_P$ is the M\"obius function on the set $X_C:=\{D \mid D \preceq C\}$ considered as a poset with a partial order given by set-inclusion. We first prove the following key result. 
	
	\begin{nothma}(Proposition \ref{proposition: P_C primitive})
		Let $C$ be a connected multi-complex. Then $P_C$ is a primitive element of $H_\mathcal{C}$. Moreover, if $C=C_1\cdot C_2 \cdots C_r$ as an element of $H_\mathcal{C}$, then $P_C=P_{C_1}\cdot P_{C_2} \cdots P_{C_r}$ as an element of $H_\mathcal{C}$. 
	\end{nothma}
	
	Next, we show that each multi-complex $C$ is a polynomial with non-negative integer coefficients of the elements $P_{D}$: 
	
	\begin{nothmb}(Proposition \ref{proposition: inversion P} and Theorem \ref{thm: main theorem P})
		Let $C$ be a multi-complex. Then we have the following:
		\[
		C=\sum_{D \preceq C}P_D
		\]
		In particular, if $\mathcal{T}$ is the set of isomorphism classes of connected multi-complexes in $\mathcal{C}$, then the set $\{P_C\}_{C \in \mathcal T}$ forms a linear basis of the space of primitive elements of $H_{\mathcal{C}}$. 
	\end{nothmb}
	
	Finally, we utilize the above results to find a cancellation-free and grouping-free antipode formula for the Hopf algebra of multi-complexes, as well as cancellation-free formulas for the transformations between the basis $\{C\}_{C \in \mathcal{T}}$ of $H_{\mathcal{C}}$ and $\{P_C\}_{C \in \mathcal{T}}$. They are given in terms of multiplicity numbers $[C:D]$, defined for every arbitrary pair of multi-complexes $C,D$, to be the number of sub-multi-complexes of $C$ isomorphic to $D$. The following statement summarizes below all these results, given in \S \ref{section: Applications}.
	
	\begin{nothmc}
		
		For any multi-complex $C$, the following hold. 
		\begin{eqnarray*}
			C & = & \sum_{\stackrel{D_1,D_2,\dots,D_t\,{\rm connected}}{n_C=n_{D_1}\sqcup n_{D_2} \sqcup \cdots \sqcup n_{D_t}}} [C:D_1\dots D_t]P_{D_1}\ldots P_{D_t} \\
			P_C & = & \sum_{\stackrel{D_1,D_2,\dots,D_t\,{\rm connected}}{n_C=n_{D_1}\sqcup \cdots \sqcup n_{D_t}}} \mu_P(D_1\sqcup\cdots \sqcup D_t,C) [C:D_1\dots D_t]D_1\cdot \ldots \cdot D_t \\
			S(C) & = & \sum_{D \preceq C} (-1)^{c_{D}}P_{D} = \sum_{\stackrel{D_1,D_2,\dots,D_t\,{\rm connected}}{n_C=n_{D_1}\sqcup n_{D_2} \sqcup \cdots \sqcup n_{D_t}}} (-1)^t[C:D_1\dots D_t]D_1\cdot \ldots \cdot D_t,
		\end{eqnarray*}
	where $c_{D}$ is the number of connected components of the sub-multi-complex $D$ of $C$. 
	\end{nothmc}
We note that the above basis $\{P_C\}_{C \in \mathcal{T}}$ satisfies the requirements of the general problem posed above; in fact, we prove that $\{P_C\}_{C \in \mathcal{T}}$ satisfies certain minimality property and a universal property among the bases satisfying two natural conditions, which one might regard as combinatorial bases. 

	
	We also recover some known facts for on the graph reconstruction conjectures using our framework. Finally, we note that our main ideas for finding such a bases, and consequently, cancellation-free formulas for the antipode, might very well be suitable to application in other combinatorial contexts. 
	
	This paper is organized as follows. In the interest of a wider audience, in \S \ref{section: Preliminaries}, we review basic definitions for the graph Hopf algebra; familiarity with Hopf algebra theory is not required. In \S \ref{Section: Hopf algebra of multi-complexes}, we define multi-complexes and the Hopf algebra of multi-complexes. In \S \ref{Section: primitive for multi-complexes}, we introduce a basis for the space of primitive elements of the Hopf algebra of multi-complexes, and study basic properties, and give our main results. In \S \ref{section: multi-complexes of dimension at most $1$}, we specialize our results to the case of multi-complexes of dimension at most $1$ and obtain several results which directly imply the such formulas for  for graphs, multigraphs and hypergraphs. In \S \ref{section: Applications}, we give a few applications of our explicit description of a primitive basis of the Hopf algebra of multi-complexes.

	\vspace{0.5cm}
	\noindent{\bf Acknowledgements:} J.J. received support from an AMS-Simons travel grant, and thanks Jaehoon Kim for pointing out that the authors' argument in the previous version for graphs can be used to obtain the same results for hypergraphs. M.C.I was supported by the Simons Collaboration Grant 637866 for several research visits to J.J.

	\section{Preliminaries} \label{section: Preliminaries}
	
	In what follows, all graphs are assumed to be finite unless otherwise stated. For a graph $G$, we let $E(G)$ be the set of edges of $G$ and $V(G)$ the set of vertices of $G$. For each subset $E$ of $E(G)$, we let $G- E$ be the graph which we obtain from $G$ by deleting edges in $E$ while we keep the same vertex set. Finally, for each finite set $A$, we let $|A|$ be the cardinality of $A$.
	
	Let $\mathcal {G}$ be the set of isomorphism classes of graphs. One can naturally impose a (commutative) monoid structure on $\mathcal{G}$ as follows: for $[G_1],[G_2] \in \mathcal{G}$, 
	\[
	[G_1]\cdot [G_2] := [G_1\sqcup G_2], 
	\]
	where $[G_i]$ is the isomorphism class of $G_i$ and $G_1 \sqcup G_2$ is the disjoint union of $G_1$ and $G_2$. In particular, the \emph{empty graph} $[\emptyset]$ becomes the identity. We will interchangeably use $[G]$ and $G$ to denote the isomorphism class of $G$ when there is no possible confusion.
	
	Let $k$ be a field of characteristic zero and $k[\mathcal{G}]$ be the monoid algebra obtained by considering $\mathcal{G}$ as a monoid as above. Clearly $k[\mathcal{G}]$ is graded by the number of vertices of graphs. 
	
	The algebra $k[\mathcal{G}]$ has a Hopf algebra structure, which we recall here. Let $E$ be a finite set and $A$ be a subset of $E$. For the notational convenience, we let $\overline{A}:=E-A$.  The coproduct $\Delta: k[\mathcal{G}] \longrightarrow k[\mathcal{G}] \otimes_k k[\mathcal{G}]$ is defined for each $G\in \mathcal{G}$ by  
	\begin{equation}\label{eq: coproduct1}
	\Delta(G):= \sum_{T \subseteq V(G)} G_T \otimes G_{\overline{T}},
	\end{equation}
	where $G_T$ is the induced subgraph (induced by $T$) of $G$. The above formula is linearly extended to $k[\mathcal{G}]$. We further define the counit $\varepsilon:k[\mathcal{G}] \to k$ by letting it be defined on the basis $\mathcal{G}$ as
	\begin{equation}\label{eq: counit}
	\varepsilon(G):=\begin{cases}
	1, \textrm{ if } V(G) = \emptyset, \\
	0, \textrm{ if } V(G) \neq \emptyset.
	\end{cases}
	\end{equation}
	
	With the above coproduct $\Delta$ and the counit $\varepsilon$, $k[\mathcal{G}]$ becomes a connected, graded bialgebra and hence a Hopf algebra, which we will denote by $H_\mathcal{G}$. This Hopf algebra is cocommutative and commutative, and by a classical Hopf algebra theorem (for example, the more general Cartier-Kostant-Milnor-Moore theorem), it is isomorphic, as a Hopf algebra, to a polynomial algebra in (necessarily) countably many variables 
	(see, for example also \cite{grinberg2014hopf} or \cite{schmitt1994incidence}). In what follows, we aim to find special explicit bases for the primitives of $k[\mathcal{G}]$ and more general combinatorial Hopf algebras.
	
	\begin{rmk}
		We note that formulas for the antipode of $H_\mathcal{G}$ have been obtained by many authors; see for instance  Humpert and Martin \cite{humpert2012incidence},  Benedetti and Sagan \cite{benedetti2017antipodes}, and  Aguiar and Ardila \cite{aguiar2017hopf}.
	\end{rmk}

\section{The Hopf algebra of Multi-complexes} \label{Section: Hopf algebra of multi-complexes}

In this section, we first introduce the notion of multi-complexes which simultaneously generalizes several classes of combinatorial objects, including graphs, hypergraphs, simplicial sets, and simplicial complexes. 
	
For a multiset $X$, we let $\textrm{supp}(X)$ be the set of elements in $X$ without repetitions. That is, a multiset $X$ can be regarded as a function $X:S\rightarrow \mathbb{N}=\{1,2,3,\dots\}$; then we denote $S=\textrm{supp}(X)$, and we say that the multiset $X$ is based on $S$. For instance, if $X=\{a,a,b,c,c,c\}$, then $\textrm{supp}(X)=\{a,b,c\}$. 

\begin{mydef}\label{definition: multi-complex}
	By a \emph{multi-complex} $C$, we mean a finite family $\{A_i\}$ of (possibly repeating) non-empty multisets based on some subset of a finite  set $n_C:=\{1,2,...,n\}$, together with a partial order $\preceq$ such that
	\begin{enumerate}
		\item
		The singletons $\{k\}$ appear among the $A_i$'s exactly once, and $\{k\}\preceq A_i$ if and only if $k$ belongs to $A_i$. 
		\item 
		If $A_i \preceq A_j$, then $A_i$ is contained in $A_j$. 
	\end{enumerate}
	The empty collection $\emptyset$ is the only multi-complex based on the empty set $\emptyset$.
\end{mydef}

\begin{mydef}
	Let $C=\{A_i\}_{i \in I}$ be a multi-complex based on $[n]:=\{1,2,\ldots,n\}$. By a \emph{sub-multi-complex}, we mean a multi-complex $D=\{A_{j}\}_{j \in J}$ for some subset $J \subseteq I$ which is closed under comparison in $C$ in the following sense:
	\[
	\textrm{ If } A_i \preceq A_j \textrm{ and } A_j \in D,\textrm{ then } A_i \in D.
	\]
\end{mydef}

\begin{myeg}[{\bf Simple Graphs}] \label{example: graph complex}
	Let $G$ be a simple graph with the set $V(G)$ of vertices and the set $E(G)$ of edges. Then, we define the multi-complex $C_G$ as $C_G=V(G)\cup E(G)$, and $\preceq$ is given for $v\in V(G)$ and $e\in E(G)$ by $\{v\}\preceq \{e\}$ if the edge $e$ contains vertex $v$. In particular, $C_G$ is based on the finite set $V(G)$. Clearly, $C_G$ uniquely determines $G$ up to isomorphism. 
\end{myeg}

\begin{myeg}[{\bf Multigraphs}]\label{example: multi-graph complex}
    Let $G$ be a multigraph (so multiple edges and loops are allowed). For each unordered pair of vertices $(a,b)$ consider a number of multisets $A_{i}(a,b)=\{a,b\}$ equal to the number of arrows between $a$ to $b$. Then the collection of multisets $\{A_i(a,b)\}_{(a,b)}\cup \{a\}_{a\in V(G)}$ is a multi-complex. For a specific example, the graph \vspace{.5cm}
    $$\xymatrix{\bullet \ar@(dr,dl)@{-}[] \ar@(ur,ul)@{-}[] 
    \ar@/^{1pc}/@{-}[r]  \ar@/_{1pc}/@{-}[r] & \bullet} 
    \vspace{.5cm} $$ yields the multi-complex $A_1=\{1\},  A_2=\{2\}, A_3=\{1,1\}, A_4=\{1,1\}, A_5=\{1,2\}, A_6=\{1,2\}$ with $\preceq$ defined minimally to fulfill the definition (i.e. there are no relations between sets $A_3,A_4,A_5,A_6$). It is clear here why in the definition of multi-complex we allow for the $A_i$ to be multisets, which can themselves repeat. 
\end{myeg}

\begin{myeg}[{\bf Hypergraphs}]\label{example: hyper-graph complex}
    Similarly as above, if  $G$ is a hypergraph with vertex set $V(G)$ and edge set $E(G)$, the same definition as in the previous example $C_G=V(G)\cup E(G)$ yields a multi-complex $C_G$.    
\end{myeg}

\begin{myeg}[{\bf Simplicial Complexes and $\Delta$-Complexes}]\label{example: simplicial complex}
Let $S$ be a finite abstract simplicial complex, and $S_n$ the set of simplices of dimension $n$. We let $C_S=\bigcup\limits_n S_n$ to be the set of all simplices, and $\preceq$ be defined as $A\preceq B$ if $A\subseteq B$ (so $A$ is a face of the simplex $B$). It is clear that the multi-complex $C_S$ uniquely recovers $S$ up to isomorphism. Similarly and more generally, a finite $\Delta$-complex $X$ can be regarded as a multi-complex by considering the poset of simplices of $X$ as follows. Let $[n]=\{1,2,...,n\}$ be the set of 0-simplices of $X$ and we let $C_X$ be the collection of all simplices of $X$. For $A,B\in C_X$ we let $A\preceq B$ if $A$ is a face of $B$. Since the structure of $X$ is combinatorially completely determined by this information, then $X$ is determined by $C_X$. (This covers also the case of multigraphs above.) 
\end{myeg}

\begin{myeg}
	Suppose that $C$ is a multi-complex based on $\{1,2,3\}$ with $A_1=\{1\}$, $A_2=\{2\}$, $A_3=\{3\}$, $A_4=\{1,2\}$, $A_5=\{1,2\}$, $A_6=\{2,3\}$, $A_7=\{1,3\}$, $A_8=\{1,2,3\}$ together with the partial order $\preceq$ given by inclusion except for that we require $A_5$ and $A_8$ are not comparable. This models the Delta-complex given by one 2-simplex (full triangle) and one additional edge between vertices $1$ and $2$, as depicted in the picture below.
	$$\xymatrix{\bullet \ar@{-}[d]\ar@(ur,ul)@{-}[d]\ar@{-}[dr]_{{}_{=}} & \\ \bullet\ar@{-}[r] & \bullet} $$
\end{myeg}

\begin{myeg}[{\bf Colored simplicial complexes}]\label{example: colored simplicial complexes}
    By a \emph{colored simplicial complex} or \emph {labeled simplicial complex} we mean a simplicial complex $S$ together with a coloring of its faces (or label attached to each of its faces), that is, a function $f:F(S)\rightarrow \mathbb{N}$, where we fix a countable set of colors $\mathbb{N}=\{0,1,2,...\}$. We note that there are no relations imposed between color of faces and their sub-simplices. We can associate a unique multi-complex $C=C(S,f)=(A_i)_i$ to each such colored simplicial complex as follows. Given $S$ and $f$, first let $n_C$ be the set of 0-simplices of $S$. For each $a\in n_C$, we add a number of new sets $A_i=\{a,a\}$ equal to $f(a)$. Intuitively, at each vertex, we add a number of "loops" that will encode the color number of that vertex. Continuing, for each face $E=(a_1,...,a_t)$ of $S$, we add a number of sets $A_j=\{a_1,...,a_t\}$ equal to $f(E)$. We let the relation between the $A_i$'s be containment. One can easily see that each colored simplicial complex yields a uniquely well defined multi-complex. 
\end{myeg}

\begin{mydef}\label{definition: order of multi-complexes}
	Let $C = \{A_i\}$ and $D=\{B_j\}$ be multi-complexes based on $[n]=\{1,...,n\}$ and $[m]=\{1,...,m\}$ respectively. A \emph{morphism} from $C$ to $D$ is a function $f:[n] \to [m]$ such that for each $i$, $f(A_i) \in \{B_j\}$ and such that $f$ preserves partial orders. 
\end{mydef}

Next, we introduce two important operations for multi-complexes, namely, disjoint union and restriction. 

\begin{mydef}
    Let $C = \{A_i\}$ and $D=\{B_j\}$ be multi-complexes based on $[n]$ and $[m]$ respectively.
	The \emph{disjoint union} $C \sqcup D$ of $C$ and $D$ is the multi-complex based on the set $[n] \sqcup [m] = [n+m]$ with the collection $\{A_i\} \sqcup \{B_j\}$ of multisets, together with a partial order induced from $C$ and $D$.
\end{mydef}

It is straightforward to check that this is indeed a multi-complex. We also define  the intersection of sub-multi-complexes which will be needed later. 

\begin{mydef}\label{definition: intersection}
Let $C=\{A_i\}$ be a multi-complex based on $[n]$, and suppose that $D=\{A_i \mid i\in F\}$ $E=\{A_i \mid i\in H\}$ are sub-multi-complexes of $C$. The \emph{intersection} $D\cap E$ is the multi-complex:
\[
D \cap E :=\{A_i  \mid i \in F\cap H\},
\]
and with partial order inherited from $C$. 
\end{mydef}

It is not difficult to note that indeed the above defines a multi-complex. 

\begin{mydef}
	Let $C=\{A_i\}$ be a multi-complex based on $[n]$. Let $X$ be a subset of $[n]$. We let the \emph{restriction} $C|_X$ be the multi-complex based on $X$ and consisting of $A_i$ in $C$ such that $\textrm{supp}(A_i)\subseteq X$, and with partial order inherited from $C$, so $A_i \preceq A_j$ in $C|_X$ if and only if $A_i \preceq A_j$ in $C$.  
\end{mydef}

One can easily see that the structure defined above is indeed a multi-complex.

\begin{myeg}\label{example: multi graph disjoint union}
	Let $G_1$ and $G_2$ be graphs. Then one has
	\[
	C_{G_1 \sqcup G_2} = C_{G_1} \sqcup C_{G_2},
	\]
	where $C_G$ is the multi-complex associated to a graph $G$.
\end{myeg}

\begin{myeg}\label{example: multi graph restriction}
	Let $C=C_G$ be the multi-complex associated to a graph $G$; in particular $C_G$ is based on the set $V(G)$. For any $X \subseteq V(G)$, one can easily see that
	\[
	C|_X=C_{G|_X}, 
	\]
	where $G|_X$ is the restriction of $G$ to $X$. 
\end{myeg}

Next, we introduce some definition useful in decomposing  a muti-complex as a disjoint union of ``connected components''.

\begin{mydef}\label{definition: path-component}
	Let $C$ be a multi-complex based on $[n]$.
	\begin{enumerate}
		\item 
		We say that $a,b \in [n]$ are \emph{path-connected} (with respect to $C$) if there exist $A_1,\dots,A_l \in C$ and $a_1,\dots,a_{l-1} \in [n]$ such that $\{a,a_1\} \subseteq \textrm{supp}(A_1)$, $\{a_1,a_2\} \subseteq \textrm{supp}(A_1)$, $\dots$, $\{a_{l-1},b\} \subseteq \textrm{supp}(A_l)$.
		\item 
		For $a \in [n]$, we let $X_a$ be the set of $b\in[n]$ which are path-connected to $a$. We define the \emph{connected component} of $a$ to be the multi-complex $C[a]=C|_{X_a}$. 
	\end{enumerate} 
\end{mydef}

\begin{mydef}
	Let $C$ be a multi-complex based on $[n]$. We say that $C$ is \emph{connected} if any two elements $a,b \in [n]$ are path-connected. 
\end{mydef}


\begin{myeg}
	Let $C_G$ be the multi-complex associated to a graph $G$. For $a \in V(G)$, we let $G_a$ be the (graph-theoretically) connected component of $G$ containing $a$. Then, the connected component of $a$ (as a multi-complex), where $V(G)$ is considered as the ground set of $C_G$, is just the multi-complex associated to $G_a$. 
\end{myeg}

Let $C$ be a multi-complex based on $E$. For any $a,b \in E$, we let $a \sim_C b$ if and only if $a$ and $b$ are path-connected with respect to $C$. Then one can easily see that $\sim_C$ is an equivalence relation directly by the definition. Fix a set $R_C$ of distinct representatives of $\sim_C$. For each $r \in R_C$, we note that $C[r]$, the connected component of $r$, is a connected multi-complex, and we clearly have the following.


\begin{lem}
	With the same notation as above, we have that
	\[
	C= \bigsqcup_{r \in R_C}C[r].
	\]
\end{lem}
\bigskip

Now, we can define the Hopf algebra $H_\mathcal{C}$ of multi-complexes by using the following recipe:
\begin{enumerate}
	\item 
	We let $\mathcal{C}$ be the set of isomorphism classes of multi-complexes. Then, $\mathcal{C}$ has a natural monoid structure via disjoint union. For each isomorphism class $[C]$ of $C$, we simply write $C$ unless there is any possible confusion. 
	\item 
	Let $H_\mathcal{C}:=k[\mathcal{C}]$ be the monoid algebra over a field $k$ of characteristic zero. We will write $C\cdot D=C\sqcup D$ (or more precisely, $[C]\cdot [D]=[C\sqcup D]$). Then $H_\mathcal{C}$ is graded; for each $C$, the grading of $C$ is the number $|n_C|$, the cardinality of the base set of $C$.
	\item 
	The coproduct on $H_\mathcal{C}$ is the (usual) sum over all partitions of $n_C$ of induced multi-complexes, that is,
	\begin{equation}\label{eq: coproduct}
	\Delta(C):= \sum_{X\sqcup Y =n_C} C|_X \otimes C|_Y, 
	\end{equation}
	where $C|_X$ is the restriction of $C$ to $X$. 
	\item 
	The counit $\varepsilon: H_\mathcal{C} \to k$ is defined on each $C \in \mathcal{C}$ as $\varepsilon(C) =1$ if $C$ is based on the empty set, and $\varepsilon(C)=0$ otherwise, then linearly extended to all of $H_\mathcal{C}$. 
\end{enumerate}

\begin{lem}\label{lemma: multi-complex Hopf}
	With the same notation as above, $H_\mathcal{C}$ is a Hopf algebra. 
\end{lem}
\begin{proof}
	We first show that $H_\mathcal{C}$ is bialgebra, that is, we show that the following diagrams commute. 
	\[
	\begin{tikzcd}[column sep=large]
	H_\mathcal{C}\otimes_kH_\mathcal{C} \arrow{r}{\Delta \otimes \id}
	&H_\mathcal{C} \otimes_kH_\mathcal{C}\otimes_kH_\mathcal{C} \\
	H_\mathcal{C} \arrow{r}{\Delta} \arrow{u}{\Delta} & H_\mathcal{C}\otimes_kH_\mathcal{C}
	\arrow{u}{\id \otimes \Delta},
	\end{tikzcd} \quad
	\begin{tikzcd}[column sep=large]
	H_\mathcal{C}\otimes_k H_\mathcal{C}\arrow{r}{\varepsilon \otimes \id}
	& k\otimes_k H_\mathcal{C} \\
	H_\mathcal{C} \arrow{r}{\id} \arrow{u}{\Delta} & H_\mathcal{C} \arrow{u}{\simeq}
	\end{tikzcd}
	\]
	The first diagram clearly commutes since we have, for any $C \in H_\mathcal{C}$ based on $[n]$, 
	\[
	(\Delta \otimes \id)\circ \Delta (C)=\sum_{X\sqcup Y \sqcup Z =[n]} C|_X \otimes C|_Y \otimes C|_Z = \Delta \circ(\Delta \otimes \id)(C). 
	\]
	For the second diagram, suppose that $C$ is based on the empty set, i.e. $C=\emptyset$. Then we have 
	\[
	(\varepsilon \otimes \id)\circ \Delta (C) = (\varepsilon \otimes \id)(C \otimes C )=1 \otimes C,  
	\]
	and hence the second diagram also commutes. Furthermore clearly, $H_\mathcal{C}$ is a connected and graded bialgebra, showing that $H_\mathcal{C}$ is a Hopf algebra. 
\end{proof}

The following statement is obvious from the above definitions of the Hopf algebra of multi-complexes.

\begin{lem}\label{lemma: sub Hopf}
    Let $\mathcal{E}$ be a set of multi-complexes which is closed under the operations of taking disjoint union and restriction. Then the subalgebra $H_{\mathcal{E}}$ of $H_{\mathcal{C}}$ generated (and spanned) by the set $\mathcal{E}$ is a Hopf subalgebra of $H_{\mathcal{E}}$. 
\end{lem}


\begin{myeg}
(1)	Let $\mathcal{C}_\mathcal{G}$ be the set of multi-complexes of the form $C_G$ for some finite graph $G$. It is obvious clear that restricting such a multi-complexes produces another multi-complex of the same form. Then, $H_{\mathcal{C}_{\mathcal G}}$ is a Hopf subalgebra of $H_{\mathcal{C}}$. This Hopf algebra is known as the Hopf algebra of graphs.  \\ 
(2) Similarly, if $\mathcal{C}_\mathcal{MG}$ and $\mathcal{C}_\mathcal{HG}$ denote the set of multi-complexes associated to multigraphs and hypergraps, respectively, as in Examples \ref{example: multi-graph complex} and \ref{example: hyper-graph complex}. Again, it is easy to see that these families of multi-graphs are closed under disjoint union and restriction, and thus the Hopf algebras $H_{\mathcal{C}_\mathcal{MG}}$ and $H_{\mathcal{C}_\mathcal{MG}}$ are Hopf subalgebra of $H_{\mathcal{C}}$. \\
(3) The similar statements work for simplicial complexes, delta complexes and colored simplicial complexes (see Examples 
\ref{example: simplicial complex} and \ref{example: colored simplicial complexes}); each of these classes of multi-complexes generate a Hopf subalgebra of $H_{C}$. To see this, for example, for colored simplicial complexes, by the above Lemma one only needs to note that the operation of taking the ``induced" complex agree: that is, if $(S,f)$ is a colored simplicial complex as in \ref{example: colored simplicial complexes}, and $X$ is  a subset of the vertices of $S$, if we let $S'$ take the colored simplicial complex obtained by retaining all the faces of $S$ supported on vertices in $X$, together with their respective colorings, and then  associate the multi-complex $C(S',f_{\vert S'})$, we obtain precisely the restriction of the multi-complex $C=C(S,f)$ to the set $X\subseteq n_C$.   
\end{myeg}
	
\section{Primitive elements of the Hopf algebra of multi-complexes} \label{Section: primitive for multi-complexes}

In this section, we obtain a base of the vector space of primitive elements of $H_\mathcal{C}$ and investigate its properties. 


\begin{notation}
	Let $C$ be a multi-complex based on $n_C$. By slight abuse of notation, we write $D \preceq C$ if $D$ is a sub-multi-complex of $C$ such that $n_D=n_C$.
\end{notation}

\begin{lem}\label{lemma: trivial lemma}
	Let $C_1,C_2$ be multi-complexes based on $X_1$ and $X_2$ respectively, and $C=C_1 \sqcup C_2$. Then we have the following:
	\begin{enumerate}
		\item 
		If $D_1 \preceq C_1$ and $D_2 \preceq C_2$, then $(D_1\sqcup D_2) \preceq C$. 
		\item 
		If $D \preceq C$, then $D=D_1\sqcup D_2$ where  $D_i=D|_{X_i}=D\cap C_i$ for $i=1,2$. Furthermore, $D_i$ consists of all $A$ in $D$ for which $\emph{supp}(A)\subseteq X_i$.
		
	\end{enumerate}	
\end{lem}
\begin{proof}
The first assertion is trivial. For the second assertion, suppose that $A_i \in D$. By the definition, either $A_i \in C_1$ (in this case $A_i \in D_1$) or $A_i \in C_2$ (in this case $A_i \in D_2$), showing that $D \subseteq D_1 \sqcup D_2$. Conversely, one can easily see from the definition that $D_1\sqcup D_2 \subseteq D$. 
\end{proof}

	
	
	


For a multi-complex $C=\{A_i\}$, let $X_C$ be the set of sub-multi-complexes $D$ of $C$ such that $D \preceq C$, considered as a poset with the partial order given by the same $\preceq$; that is, for $D,E$ such that $D\preceq C$ and $E\preceq C$, we have $D \preceq E$ if and only if $D$ is a sub-multi-complex of $E$. Let $\mu_P$ be the M\"obius function on $X_C$. For each multi-complex $C$, we define the following element in $H_\mathcal{C}$:
\begin{equation}\label{eq: P_C}
P_C:=\sum_{D \preceq C}\mu_P(D,C)D,
\end{equation} 
where the sum runs over all sub-multi-complexes $D$ of $C$ such that $n_D=n_C$. Obviously, this is well defined as an element of $H_\mathcal{C}$ and depends only of the isomorphism class of $C$.

\begin{lem}\label{lemma: mobius1}
	Let $C_1$ and $C_2$ be multi-complexes and $C=C_1 \sqcup C_2$. For any sub-multi-complex $D$ of $C$, we have that
	\begin{equation}\label{eq: mobius1}
	\mu_P(D,C) = \mu_P(D_1,C_1)\cdot \mu_P(D_2,C_2),
	\end{equation}
	where $D=D_1\sqcup D_2$ as in Lemma \ref{lemma: trivial lemma}. 
\end{lem}
\begin{proof}
	Recall that for a poset $P$ and $x \preceq y \in P$, a chain from $x$ to $y$ is a sequence $T=(x_0,x_1,\dots,x_k)$ in $P$ such that
	\[
	x=x_0 \prec x_1 \prec \cdots \prec x_k=y
	\]
	The length $\ell(T)$ of a chain $T$ is then defined as $\ell(T)=k$.  The M\"{o}bius function $\mu$ for $P$ is given as follows: (see, \cite[Theorem 1.2.]{crapo2008primitive})
	\[
	\mu(x,y)=\sum (-1)^{\ell(T)}, 
	\]
	where the sum runs over all chains $T$ from $x$ to $y$ in $P$. 
	
	Now, let $T_1=(D_1,a_1,\dots,a_n,C_1)$ and $T_2=(D_2,b_1,\dots,b_m,C_2)$ be chains. Then we define the following chain $T_1 \sqcup T_2$:
	\[
	D_1\sqcup D_2 \prec a_1 \sqcup b_1 \prec a_2 \sqcup b_1 \prec \dots \prec a_n \sqcup b_1  \prec a_n \sqcup b_2  \prec a_n \sqcup b_3  \prec  \cdots \prec a_n \sqcup b_m \prec C_1 \sqcup C_2.
	\]
	Then, we have $\ell(T_1 \sqcup T_2) = \ell(T_1)+ \ell(T_2) -2$. Also, one can easily observe that $T_1'\sqcup T_2' = T_1 \sqcup T_2$ if and only if $(T_1',T_2')=(T_1,T_2)$. Furthermore, any chain $T$ from $D$ to $C$ arises in this way. In particular, we have that
	\[
	\mu_P(D,C) = \sum (-1)^{\ell(T)} = \sum (-1)^{\ell(T_1 \sqcup T_2)} = \sum (-1)^{\ell(T_1)+ \ell(T_2)-2} =\sum (-1)^{\ell(T_1)+ \ell(T_2)}
	\]
	\[
	=\left(\sum (-1)^{\ell(T_1)}\right)\left(\sum (-1)^{\ell(T_2)}\right) = \mu_P(D_1,C_1)\mu_P(D_2,C_2).
	\]
\end{proof}

We note that the above result can also be obtained by observing that the poset $P$ of sub-multi-complexes of $C$ is isomorphic to $P_{1}\times P_{2}$, where $P_{i}$ is the poset of sub-multi-complexes of $D_i$. Furthermore, this implies that there is an isomorphism of incidence algebras of posets $I(P)\cong I(P_1)\otimes I(P_2)$. Via this isomorphism, the zeta function $\zeta_P$ corresponds to the tensor $\zeta_{P_1}\otimes \zeta_{P_2}$, and this implies the desired result (since the M\"{o}bius function is the inverse element of the zeta function of a poset).

\begin{pro}\label{proposition: P_G in gen}
	Let $\mathcal{C}$ be the set of isomorphism classes of multi-complexes, considered as a monoid (with the product given by the disjoint union). Then the function
	\[
	P:\mathcal{C} \longrightarrow H_\mathcal{C}, \quad C \mapsto P_C
	\]
	is multiplicative. That is, $P(C_1\cdot C_2)=P(C_1) \cdot P(C_2)$. \footnote{Strictly speaking, here $C_1$ and $C_2$ are isomorphism classes of multi-complexes in $\mathcal{C}$.}
\end{pro}
\begin{proof}
	Let $C=C_1 \sqcup C_2$ (the disjoint union of $C_1$ and $C_2$). For each sub-multi-complex $D$ of $C$, we let $D_i=D \cap C_i$ for $i=1,2$. From Lemmas \ref{lemma: trivial lemma} and \ref{lemma: mobius1}, we have
	\[
	P(C)=\sum_{D \preceq C}\mu_P(D,C) D = \sum_{\substack{D_1\preceq C_1\\ D_2\preceq C_2}} \mu_P(D_1,C_1)\mu_P(D_2,C_2) D_1D_2
	\]
	\[
	=\sum_{D_2 \preceq C_2} \left( \sum_{D_1 \preceq C_1} \mu_P(D_1,C_1) D_1\right) \mu_P(D_2,C_2)D_2 = P_{C_1}P_{C_2}=P(C_1)P(C_2). 
	\]
\end{proof}

From Proposition \ref{proposition: P_G in gen}, if $C=C_1\cdot C_2 \cdots C_k$ (as an element of $H_\mathcal{C}$) for some connected multi-complexes $C_i$, then 
\[
P_C=\prod_{i=1}^k P_{C_i}.
\]

\begin{pro}\label{proposition: P_C primitive}
	Let $C$ be a connected multi-complex. Then $P_C$ is a primitive element in $H_\mathcal{C}$. 	
\end{pro}
\begin{proof}
	For a multi-complex $F$, we let $n_F$ be the set on which $F$ is based. Since the coproduct is a linear map, with the notation as in \eqref{eq: coproduct}, we have
	\begin{align}\label{eq: connected primitive1}
	\begin{aligned}
	\Delta(P_C)&=\sum_{F\preceq C}\mu_P(F,C) \sum_{X \sqcup Y=n_F} F|_X \otimes F|_Y =\sum_{X \sqcup Y=n_F}\sum_{F\preceq C}\mu_P(F,C)F|_X \otimes F|_Y 
	\\
	&=\sum_{X\sqcup Y=n_C}\sum_{\substack{E,H:\\ n_E=X,n_H=Y}}\left(\sum_{\substack{F:\\ F|_X=E,F|_Y=H}} \mu_P(F,C) E \otimes H \right)
	\end{aligned}
	\end{align}
	For $X \sqcup Y = n_C$, we let
	\[
	\alpha(E,H):=\sum_{\substack{F:\\ F|_X=E,F|_Y=H}} \mu_P(F,C).
	\]
	We claim that if $X \neq \emptyset$ and $Y \neq \emptyset$, then $\alpha(E,H)=0$. Indeed, we fix $X \sqcup Y =n_C$ with $X \neq \emptyset$ and $Y \neq \emptyset$. For any pair $E,H$ with $n_E=X$, $n_H=Y$ and $E \sqcup H \preceq C$, we have
	\begin{equation}
	\alpha(E,H)=\sum_{\substack{F:\\ F|_X=E,F|_Y=H}} \mu_P(F,C) =\sum_{\substack{E\sqcup H \preceq F \preceq C:\\ F|_X=E,F|_Y=H}} \mu_P(F,C).
	\end{equation}
	Note that if $E=C|_X$ and $H=C|_Y$, then for any $F$ such that $E\sqcup H \preceq F \preceq C$ we have $F|_X=E$ and $F|_Y=H$. Thus, we have
	\[
	\alpha(E,H)=\sum_{E\sqcup H\preceq F \preceq C}\mu_P(F,C)=0
	\]
	since the interval $[E\sqcup H,C]$ is non-trivial by the connectivity of $C$ (i.e. $C\neq E\sqcup H$). Now, let $E\sqcup H\preceq C$ be arbitrary with $n_E=X$ and $n_H=Y$. We have again by the definition of $\mu$ and the connectivity of $C$, 
	\[
	\sum_{E\sqcup H \preceq F \preceq C} \mu_P(F,C)=0.
	\]
	But, now we have
	\begin{align}\label{eq: connected primitive2}
	\begin{aligned}
	0&=\sum_{E\sqcup H \preceq F\preceq C}\mu_P(F,C)  =\sum_{\substack{E',H':\\n_{E'}=X,n_{H'}=Y\\ E\preceq E',H \preceq H',E'\sqcup H' \preceq C}}\left(\sum_{\substack{F:\\ F|_X=E',F|_Y=H'}} \mu_P(F,C)  \right)
	\\
	&=\left(\sum_{\substack{E',H':\\ E\preceq E', H\preceq H'\\n_{E'}=X,n_{H'}=Y\\ E\sqcup H \prec E'\sqcup H' \preceq C}} \alpha(E',H') \right) + \alpha(E,H)
	\end{aligned}
	\end{align}
	Using the above formula \eqref{eq: connected primitive2}, an easy induction on the length of the interval $[E\sqcup H, C]$ shows  that $\alpha(E,H)=0$ (since  $\ell([E'\sqcup H',C])<\ell([E\sqcup H,C])$ for $E',H'$ as above). 
	Now, we have from \eqref{eq: connected primitive1}, 
	\begin{equation}\label{eq: connected primitive3}
	\Delta(P_C)=\sum_{X\sqcup Y=n_C}\sum_{\substack{E,H:\\ n_E=X,n_H=Y}}\alpha(E,H) E \otimes H
	\end{equation}
	But, we have shown that $\alpha(E,H)=0$ except when $X=n_E=n_C$ or $Y=n_H=n_C$, when obviously $\alpha(C,\emptyset)=\alpha(\emptyset,C)=1$. In particular, we have
	\[
	\Delta(P_C)=\sum_{F \preceq C}\mu_P(F,C)F \otimes \emptyset + \sum_{F \preceq C} \mu_P(F,C) \emptyset \otimes F = P_C \otimes \emptyset + \emptyset \otimes P_C, 
	\]
	showing that $P_C$ is primitive. 
\end{proof}

\begin{lem}\label{lemma: connected inversion1}
	Let $C$ be a connected multi-complex. Then we have the following:
	\[
	C= \sum_{D\preceq C} P_D, 
	\]
	where the sum runs over all sub-multi-complexes $D$ of $C$ such that $n_D=n_C$.
\end{lem}
\begin{proof}
	Let $X_C$ be the set of sub-multi-complexes $D$ of $C$ such that $n_D=n_C$ (i.e. $D \preceq C$) and $H_\mathcal{C}$ be the Hopf algebra of multi-complexes, as before. Consider the following functions:
	\[
	g: X_C \longrightarrow H_\mathcal{C}, \quad D \mapsto D, 
	\]
	\[
	f: X_C \longrightarrow H_\mathcal{C}, \quad D \mapsto P_D.
	\]
	By the definition of $P$, for $D \in X_C$, we clearly have
	\[
	f(D)=P_D=\sum_{E \preceq D} \mu_P(E,D)D=\sum_{E \preceq D} \mu_P(E,D)g(D).
	\]
	It now follows from the M\"obius inversion formula (for the incidence algebra of the poset $(P,\preceq)$ of sub-multi-complexes $D$ of $C$ with $n_D=n_C$) that
	\[
	g(D)=\sum_{E \preceq D} \zeta(E,D)f(E)=\sum_{E \preceq D} P_{E}, 
	\]
	where $\zeta$ is the zeta function of the incidence algebra of $P$ $(\zeta(X,Y)=1$ for all $X \preceq Y$). In particular, when $D=C$, we have that
	\[
	C=\sum_{E \preceq C}P_E. 
	\]
\end{proof}

\begin{pro}\label{proposition: inversion P}
	Let $C$ be a multi-complex. Then we have the following:
	\[
	C= \sum_{D \preceq C} P_D. 
	\]
\end{pro}
\begin{proof}
	We can uniquely write $C=C_1\cdots C_k$ for some  (necessarily connected) path-components $C_i$. It follows from Lemma \ref{lemma: connected inversion1} that for each $C_i$, we have 
	\[
	C_i=\sum_{D_i \preceq C_i} P_{D_i}. 
	\]
	In particular, we have that
	\[
	C=C_1 \cdots C_k =\prod_{i=1}^k(\sum_{D_i \preceq C_i}P_{D_i}) = \sum_{D \preceq C} (\prod_{i=1}^kP_{D\cap C_i}),
	\]
	since each $D\preceq C$ is uniuely written as $D=D_1\sqcup ... \sqcup D_k$ with $D_i=D\cap C_i=D|_{n_{C_i}}\preceq C_i$.  Furthermore, since in this case, $D=D_1\cdot D_2\cdots D_k$ in the monoid $\mathcal{C}$, it follows from Proposition \ref{proposition: P_G in gen} that 
	\[
	\sum_{D \preceq C} (\prod_{i=1}^kP_{D\cap C_i})=\sum_{D \preceq C} P_{D}.
	\]
\end{proof}

\begin{cor}\label{corollary: primitive polynomial multi-complex}
Let $C$ be a multi-complex and $\mathcal{T}$ be the set of isomorphism classes of connected multi-complexes. Then $C$ can be written (as an element of $H_\mathcal{C}$) as a polynomial with non-negative integer coefficients in the elements $\{P_C\}$ for $ C\in \mathcal{T}$.  
\end{cor}
\begin{proof}
	From Proposition \ref{proposition: inversion P}, we have
	\begin{equation}\label{eq: product11}
	C=\sum_{D \preceq C} P_{D}=\sum_{j=0}^r P_{D_j}=\sum_{j=0}^r \left(\prod_{t=0}^{w_j}  P_{D_{j_t}}\right),
	\end{equation}
where $D_{j_t}$ are connected components of $D_j$. Now, regrouping terms appropriately in  \eqref{eq: product11} gives the desired result (alternatively, we can write $C=C_1\cdot...\cdot C_k$ and apply Proposition \ref{proposition: inversion P} to each $C_i$). 
\end{proof}

The following theorem gives a basis of the vector space of primitive elements of $H_{\mathcal{C}}$.
We need the following observation first. For any multi-complex $C=\{A_i\}_{i=1,...,t}$, let the size of $C$ be the number $|C|=t$. By the definition of the disjoint union for multi-complexes, we obviously have
\[
|C\sqcup D|=|C|+|D|.
\]
Therefore we have a grading on $H_\mathcal{C}$ by size. Note that we also have that if $C\preceq D$ but $C\neq D$ then $|C|<|D|$; consequently, if $C\preceq D$ and $|C|=|D|$ then $C=D$. 

\begin{mythm}\label{thm: main theorem P}
	Let $\mathcal{T}$ be the set of isomorphism classes of connected multi-complexes in $\mathcal{C}$. The set $\{P_C\}_{C \in \mathcal T}$ forms a basis of the vector space of primitive elements of $H_{\mathcal{C}}$. 
\end{mythm}
\begin{proof}
	Let $V$ be the vector space of primitive elements in $H_\mathcal{C}$. We first show that the set $\{P_C\}_{C \in \mathcal{T}}$ is linearly independent. Suppose that we have 
	\begin{equation}\label{eq: linearly independent1}
	a_1P_{C_1} + \cdots + a_nP_{C_n} = 0. 
	\end{equation}
	where no two multi-complexes $C_i$ and $C_j$ for $i\neq j$ are isomorphic. Suppose that $a_i \neq 0$ for some $a_i$. From the definition, we have that
	\[
	P_{C_i} = \sum_{D_{i_j} \preceq C_i} \mu_P(D_{i_j},C_i)D_{i_j}. 
	\]
	 Consider the grading of $H_{\mathcal{C}}$ by size as above. Substituting $P_{C_i}$ above into equation \eqref{eq: linearly independent1}, using also the previous remark, we see that the element of highest size of the left hand side of that equality is of the form $\sum\limits_{i\in F}a_i{C_i}$ for some non-empty subset $F$ of $\{1,\dots,n\}$. For that to equal zero, we must have two indices $i$ and $j$ such that $C_i=C_j$. But, by definition, this implies that $P_{C_i}=P_{C_j}$, giving a contradiction. Therefore $\{P_C\}_{C \in \mathcal{T}}$ is linearly independent. 
	
	Finally, we note that $\{P_C\}_{C \in \mathcal{T}}$ spans $V$. From Proposition \ref{proposition: inversion P}, we have that the elements $P_C$ generates $H_{\mathcal{C}}$; but a set of primitives which generate the whole Hopf algebra must span the space of primitives. 
	(if $p\in V$, then write $p=f((P_{C_i})_{i=1,...,t})$, a polynomial in the $P_{C}$; this is primitive only when it is a linear polynomial with zero constant term.) 
\end{proof}

\section{Multi-complexes of dimension 1}\label{section: multi-complexes of dimension at most $1$}

We introduce now a natural notion of dimension for multi-complexes and examine here the case of multi-complexes of dimension 1, where formulas become much more explicit. In particular, they will apply for simple graphs, multigraphs, and hypergraphs.

\begin{mydef}
Let $C=\{A_i\}$ be a multi-complex. We define the \emph{dimension} $\dim(C)$ of $C$ to be the maximum number $d$ such that there is a chain $A_{i_0}\prec A_{i_1}\prec \dots \prec A_{i_d}$ in $C$ of length $d$. 
\end{mydef}

\begin{mydef}
Let $C=\{A_i\}_{i=1,\dots,t}$ be a multi-complex.\\
(i) If $S=\{A_i\}_{i\in F}$ is a subfamily of $C$ with $F\subseteq \{1,\dots,t\}$. We define the sub-multi-complex $\langle S\rangle$ generated by $S$ to be the smallest sub-multi-complex of $C$ containing $S$. \\
(ii) For each face $A_i$ of $C$, we define the dimension of $A_i$ to be the number
\[
\dim(A_i)=\dim(\langle \{A_i\}\rangle),
\] 
that is, the dimension of the sub-multi-complex of $C$ which is generated by $A_i$.
\end{mydef}

We note that the dimension of a face $A_i$ is then the length of a maximal chain in $C$ which terminates at $A_i$. 
With these definitions, it is immediate to see that simple graphs, multigraphs and hypergraphs, as well as 1-dimensional $\Delta$-complexes, are 1-dimensional when regarded as multi-complexes. Furthermore, any strict relation in a 1-dimensional multi-complex $C$ must be of the form $A_i\prec A_j$ where $A_i$ consists of a ``vertex'' $A_i=\{k\}$. For a multi-complex $C$ of dimension at most 1, we define the ``edge'' set of $C$ to be the family $E(C)=\{A_i\}_{i\in U}$ consisting of those $A_i$ of dimension 1; equivalently, $A_i$ such that $A_i\neq \{k\}$ for all $k$ in the base set of $C$. For notational purposes, we allow here the case of dimension zero, when $E(C)=\emptyset$ (evidently, the multi-complexes of dimenison 0 reduce to plain sets). 

Let $E$ be a subfamily of $E(C)=\{A_i\}_{i\in U}$, so $E=\{A_i\}_{i\in F}$ with $F\subseteq U$. Then the family 
\[
n_C \cup \{A_i\}_{i\in U-F}
\]
forms a sub-multi-complex of $C$, which we denote simply by $C-E$. Indeed, this is the case due to the fact that there are no order relations between any such $A_i$ for $i\in U$ (so for $A_i$ in $E(C)$). Furthermore, we have $C-E\preceq C$ (they are based on the same set), and every sub-multi-complex of $C$ which is based on $n_C$ is of this form. 

It is also easy to note that sub-multi-complexes of a multi-complex of dimension $\leq 1$ (or more generally, dimension $d$) have dimension $\leq 1$ (respectively $\leq d$). Finally, let us note that the poset of sub-multi-complexes $\{D\mid D\preceq C\}$ is isomorphic to the poset $\mathcal{P}(U)$, the power set of $U$. Since in this case, for subsets $F_1\subseteq F_2\subseteq U$, the M\"obius function is given by $\mu(F_1,F_2)=(-1)^{|F_2-F_1|}$, we have that for any two sub-multi-complexes $C-E\subseteq C-E'$ (with $E'$ a subfamily of $E$), the M\"obius function $\mu_C$ is given by
$$\mu_C(C-E,C-E')=(-1)^{|E-E'|}.$$
We therefore have the following theorem, which summarizes all the results of the previous section for the case of multi-complexes of dimension 1. Its proof is obvious and is based on the above observations.

\begin{mythm}\label{theorem: multi-complex of dimension 1}
Let $\mathcal{U}$ be the set of multi-complexes $C$ of dimension at most 1, and let $H_{\mathcal{U}}$ be the Hopf subalgebra of $H_{\mathcal{C}}$ spanned by all multi-complexes $C$ of dimension at most 1. Then: 
\begin{enumerate}
    \item 
    For each $C$ in $\mathcal{U}$, we have  
    $$P_C=\sum\limits_{E\subseteq E(C)}(-1)^{|E|}C$$
    \item 
    For $C,D$ in $\mathcal{U}$ we have $P_{C\sqcup D}=P_C\cdot P_D$ in $H_\mathcal{U}$, and the collection $\{P_C \mid C=\,{\rm connected}\}$ forms a linear basis in the space of primitive elements of $H_\mathcal{U}$.
    \item 
    Each $C$ in $\mathcal{U}$ can be expressed in terms of the $P_C$'s as 
        $$C=\sum\limits_{E\subseteq E(C)}P_C$$
    \item 
    Each $C$ is a polynomial with non-negative integer coefficients of the elements $\{P_D\}$ for $D\preceq C$. 
\end{enumerate}
\end{mythm}

We note that this theorem recovers, in particular, formulas obtained before for graphs by Aguiar and Mahajan in \cite[\S 9.4]{aguiar2013hopf}.

We give here a few examples to illustrate the above formulas on primitive elements. For this, we consider the special case of multi-complexes coming from graphs. As mentioned earlier, a graph $G$ uniquely determines a multi-complex $C_G$, and conversely if $C_G=C_H$, then $G=H$ (up to isomorphism). Let $\mathcal{C}_\mathcal{G}$ be the set of isomorphism classes of multi-complexes of the form $C_G$. Then, there is a natural one-to-one correspondence between $\mathcal{C}_\mathcal{G}$ and the set $\mathcal{G}$ of isomorphism classes of graphs. In what follows, we will identify $\mathcal{C}_\mathcal{G}$ and $\mathcal{G}$. 

\begin{myeg}\label{example: k3-1}
		Let $G=K_3$. Then we have, 
		\[
		P_G=G - 3~(\raisebox{0pt}{\includegraphics[width=0.03\linewidth, valign=c]{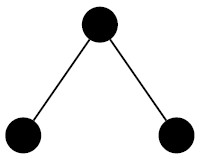}}) + 3~(\raisebox{0pt}{\includegraphics[width=0.008\linewidth, valign=c]{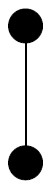}})(\raisebox{0pt}{\includegraphics[width=0.008\linewidth, valign=c]{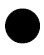}}) -(\raisebox{0pt}{\includegraphics[width=0.008\linewidth, valign=c]{w44}})^3.
		\]
	\end{myeg}

	\begin{myeg}\label{example: wedge}
		Let $G=\raisebox{0pt}{\includegraphics[width=0.05\linewidth, valign=c]{w11}}$\\
		\vspace{0.2cm}
		
		Then we have the following. 
		
		\[\raisebox{0pt}{\includegraphics[width=0.8\linewidth, valign=c]{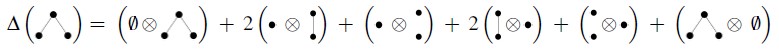}}\]
		
		Let $G_1=\raisebox{0pt}{\includegraphics[width=0.01\linewidth, valign=c]{w22}}$\hspace{0.2cm} and $G_2=\raisebox{0pt}{\includegraphics[width=0.007\linewidth, valign=c]{w44}}$. Then we have that
		\[
		P_G=G - 2G_1G_2+G_1^3.
		\]
		Furthermore, we have that 
		\[
		\Delta(G)=\emptyset \otimes G +2G_2 \otimes G_1 + G_2 \otimes G_2^2 +2G_1 \otimes G_2 + G_2^2\otimes G_2 +G\otimes \emptyset.
		\]
		\[
		\Delta(G_1G_2)=\emptyset \otimes G_1G_2 +2G_2 \otimes G_2^2 + G_2 \otimes G_1+2G_2^2\otimes G_2 +G_1 \otimes G_2 + G_1G_2 \otimes \emptyset.
		\]
		\[
		\Delta(G_2^3)=\emptyset \otimes G_2^3 + 3 G_2 \otimes G_2^2 + 3G_2^2\otimes G_2 + G_2^3 \otimes \emptyset.
		\]
		Hence, we have 
		\[
		\Delta(P_G)=\Delta(G) -2\Delta(G_1G_2) +\Delta(G_2^3)
		\]
		\[
		=\emptyset \otimes P_G + P_G \otimes \emptyset,  
		\]
		showing that $P_G$ is a primitive element. Furthermore, in this case, we have
		\[
		P_{G_1}=G_1 - G_2^2, \quad P_{G_2}=G_2. 
		\]
		It follows from Theorem \ref{theorem: multi-complex of dimension 1} that
		\[
		G=P_G+2P_{G_1}P_{G_2}+P_{G_2}^3. 
		\]
		In fact, we have that
		\[
		P_G +2P_{G_1}P_{G_2}+P_{G_2}^3 =(G -2G_1G_2+G_2^3)+2(G_1-G_2^2)G_2+G_2^3
		\]
		\[
		=G -2G_1G_2+G_2^3+2G_1G_2-2G_2^3+G_2^3=G.
		\]
	\end{myeg}
	
\begin{myeg}
    We consider here also an example coming from multigraphs. Let $G$ be the multigraph 
    \vspace{.5cm}
    $$\xymatrix{\bullet \ar@(dl,ul)@{-}[] \ar@/^{.5pc}/@{-}[r]\ar@/_{.5pc}/@{-}[r] & \bullet}$$
\vspace{.5cm}
 Then \begin{eqnarray*}P_G & = & \left(\,\,\,\,\,\,\,\,\,\,\,\,\,\, \xymatrix{\bullet \ar@(dl,ul)@{-}[] \ar@/^{.5pc}/@{-}[r]\ar@/_{.5pc}/@{-}[r] & \bullet } \right) - \left( \xymatrix{\bullet  \ar@/^{.5pc}/@{-}[r]\ar@/_{.5pc}/@{-}[r] & \bullet} \right)
 - 2\left(\,\,\,\,\,\,\,\,\,\,\,\,\,\, \xymatrix{\bullet \ar@(dl,ul)@{-}[] \ar@{-}[r] & \bullet} \right) \\ & & \\
& &  + \left(\,\,\,\,\,\,\,\,\,\,\,\,\,\, \xymatrix{\bullet \ar@(dl,ul)@{-}[]  & \bullet} \right) 
 + 2\left( \xymatrix{\bullet  \ar@/^{.5pc}/@{-}[r]\ar@/_{.5pc}/@{-}[r] & \bullet} \right)
 -\left( \xymatrix{\bullet  & \bullet} \right)
 \end{eqnarray*}
\end{myeg}

\begin{rmk}
When we specialize the Hopf algebra of multi-complexes $H_\mathcal{E}$ to any of the following specific classes $\mathcal{E}$ of combinatorial objects -  graphs, hypergraphs, multigraphs, simplicial complexes etc.  - the sub-multi-complexes of a multi-complex $C_X$ that corresponds to an object $X$ in $\mathcal{E}$ (as shown in our examples for each class) bijectively correspond to sub-objects of $X$ in $\mathcal{E}$. Furthermore, as noted before, the notion of induced sub-complex specializes to induced subobject in all of these classes. Hence, our previous formulas for the bases of primitive elements as well as the grouping and canellation free antipode formulas will carry over to all of these Hopf subalgebras $H_\mathcal{E}$ of $H_\mathcal{C}$.   
\end{rmk}

\section{Applications}\label{section: Applications}

In this section, we let $\mathcal{C}$ be the set of isomorphism classes of multi-complexes, and $\mathcal{T} \subseteq \mathcal{C}$ be the set of isomorphism classes of connected multi-complexes as before. We list several applications of our explicit description of the basis $\{P_C\}_{C \in \mathcal{T}}$. 
In the first subsection, we provide the cancellation and grouping-free antipode formula for the Hopf algebra of multi-complexes, and specialize the result to the graph Hopf algebra. 
In the second subsection, we prove that our basis $\{P_C\}_{C \in \mathcal{T}}$ is a minimal basis for the space of primitives in a certain sense, and show that $\{P_C\}_{C \in \mathcal{T}}$ satisfies a kind of universal property. Finally, in the third subsection, we list some applications to the graph reconstruction conjectures. We expect that a similar argument can be used for special cases of other reconstruction conjectures. 

\subsection*{Cancellation and grouping-free formulas} 

We first use our previous results to derive the antipode formula for $H_\mathcal{C}$. For a multi-complex $D$, we let  $c_{D}$ to be the number of connected components of $D$.  Let $C$ be a multi-complex and $D \preceq C$. 

\begin{pro}\label{proposition: antipode}
		The antipode $S(C)$ of $C\in H_{\mathcal{C}}$ can be computed as follows:
		\begin{equation}\label{eq: antipode}
		S(C)=\sum_{D \preceq C} (-1)^{c_{D}}P_{D}
		\end{equation}
	\end{pro}
	\begin{proof}
		From Proposition \ref{proposition: inversion P}, we have that
		$C=\sum\limits_{D \preceq C}P_{D}$, and by taking the antipode map $S$, we have that
		\begin{equation}\label{eq: antipodeformula}
		S(C)=S(\sum_{D \preceq C}P_{D}) = \sum_{D \preceq C} S(P_{D}).
		\end{equation}
		For each multi-complex $D$, we can uniquely write $D=D_1\cdots D_n$ for $n=c_{D}$. Then we have, 
		\[
		P_{D}=P_{D_1}\cdots P_{D_n}. 
		\]
		In particular, since each $P_{D_i}$ is primitive by Proposition \ref{proposition: P_C primitive}, 
		\begin{equation}
		S(P_{D})=S(P_{D_1}\cdots P_{D_n})=S(P_{D_1}) \cdots S(P_{D_n})=(-P_{D_1})\cdots (-P_{D_n})
		= (-1)^{c_{D}}P_{D}. 
		\end{equation}
		Therefore, we have that
		\[
		\eqref{eq: antipodeformula} = \sum_{D\preceq C} (-1)^{c_{D}}P_{D}.
		\]
\end{proof}




	\begin{myeg}\label{example: antipode}
		Let $G$, $G_1$, and $G_2$ be the graphs as in Example  \ref{example: wedge}. Then, following our formula, we have that
		\begin{align}
		\begin{aligned}
	S(G) &=  2P_{G_1\cdot G_2} - P_{G_2^3} -P_G = 2(G_1\cdot G_2 -G_2^3) - G_2^3 - (G - 2G_1\cdot G_2 + G_1^3) \\
	& =    4G_1\cdot G_2 -G - G_1^3 - 3G_2^3. 
		\end{aligned}
			\end{align}
	\end{myeg}
	
	As Example \ref{example: antipode} shows our antipode formula is not cancellation-free. Similarly, our formulas expressing each multi-complex as a linear combination of primitives are also not cancellation-free. However, we can repackage these to make them cancellation-free. To this end, we need to introduce some notation. Let $C$ and $D$ be multi-complexes. We let $[C:D]$ be the number of injective morphisms from $D$ to $C$, divided by the cardinality of $\textrm{Aut}(D)$. Equivalently, $[C:D]$ is the number of non-equivalent embeddings of $D$ into $C$, that is, the number of sub-multi-complexes of $C$ which are isomorphic to $D$. 
	
	\begin{myeg}
		Let $G_1$ and $G_2$ be graphs. Let $C_{G_1}$ and $C_{G_2}$ be the multi-complexes associated to $G_1$ and $G_2$ respectively. Then one can easily see that $G_1$ is a subgraph of $G_2$ if and only if $C_{G_1}$ is a sub-multi-complex of $C_{G_2}$. Furthermore, $[G_2:G_1]=[C_{G_2}:C_{G_1}]$, where $[G_2:G_1]$ is the number of subgraphs of $G_2$ which are isomorphic to $G_1$.
	\end{myeg}
	
	
	Given a multi-complex $C$ and connected multi-complexes $D_1,D_2,\dots,D_t$, in terms of the multiplication of $H_{\mathcal{C}}$ we can write
	\[
	[C:D_1\cdots D_t]=[C:D_1\sqcup\dots\sqcup D_t],
	\]
i.e., the numbers of sub-multi-complexes of $C$ isomorphic to $D_1\dots D_t=D_1\sqcup\dots\sqcup D_t$. Recall that $H_\mathcal{C}$ is graded, where $\deg(C)=|n_C|$ for each $C \in H_\mathcal{C}$. We can now re-write the formulas for $P_C$ as well as the formula giving $C$ as a polynomial in the primitives $P_D$ as follows
	
	\begin{eqnarray}\label{eq 1}
	P_C & = & \sum_{D \preceq C} \mu_P(D,C)D \\
	& = &   \sum_{\stackrel{[D_1],[D_2],\dots,[D_t]\,{\rm connected}}{\deg(C)=\sum_i \deg(D_i)}} \mu_P(D_1\sqcup \cdots \sqcup D_t,C) [C:D_1\cdots D_t]D_1 \cdots D_t 
	\end{eqnarray}
	
	\noindent where the second sum ranges over all t-uples of {\it isomorphism classes} $[D_1],\dots,[D_t]$ of connected multi-complexes $D_1,\dots,D_t$  such that 
	$
	\deg(C)=\sum_i \deg(D_i)
	$
	or equivalently, such that $C$ and $D_1\cdots D_t$ have the same number of vertices (then, whenever an embedding exists, they will automatically be based on the same set).  Similarly, we obtain 
	
	\begin{eqnarray}\label{eq 2}
	C & = &  \sum_{D \preceq C} P_{D} = \sum_{\stackrel{[D_1],[D_2],\dots,[D_t]\,{\rm connected}}{\deg(C)=\sum_i \deg(D_i)}} [C:D_1\cdots D_t]P_{D_1\cdots D_t} \\
	& = & \sum_{\stackrel{[D_1],[D_2],\dots,[D_t]\,{\rm connected}}{\deg(C)=\sum_i \deg(D_i)}} [C:D_1\cdots D_t]P_{D_1}\cdots P_{D_t}\label{eq: cancellation-free-antipode}
	\end{eqnarray}
	
	The last two sums above again are over isomorphism classes of multi-complexes, and thus are cancellation and grouping free. 
	Now, we deduce the cancellation-free formula for the antipode; using equations (\ref{eq: antipode}) and (\ref{eq: cancellation-free-antipode}), we have
	
	$$S(C)=\sum_{D \preceq C} (-1)^{c_{D}}P_{D} = \sum_{\stackrel{[D_1],[D_2],\dots,[D_t]\,{\rm connected}}{\deg(C)=\sum_i \deg(D_i)}} (-1)^t[C:D_1\cdots D_t]D_1\cdots D_t$$
	
\begin{rmk}\label{remark: cancellation-free graphs}
Specializing the above formulas to multi-complexes of dimension at most $1$, we obtain, in particular, cancellation and grouping free formulas for simple graphs, multigraphs, and hypergraphs. 
\end{rmk}	
	
We remark that this formula expresses the antipode of any multi-complex as a polynomial of the algebraic basis given by connected multi-complexes, and such a formula  is obviously unique (the coefficients are uniquely determined, since $H_{\mathcal{C}}$ is a polynomial algebra in the elements of $\mathcal{T}$). Thus, one can regard the multiplicities $[C:D]=[C:D_1\cdots D_t]$ as having a special meaning in the Hopf algebra of multi-complexes.

\subsection*{Minimal and Universal Properties}	
	
	We now prove that the above defined basis of the space of primitives $\{P_C\}_{C\in \mathcal{T}}$ (where $\mathcal{T}$ is the set of isomorphism class of connected multi-complexes) satisfies a certain minimality and uniqueness property. 
	
	For two multi-complexes $C,D$  we will write, by abuse of notation, $D\subseteq C$ if $D$ is a sub-multi-complex of $C$ \footnote{The notation $D \subseteq C$ differs from $D \preceq C$; the latter means that $D$ is a sub-multi-complex such that $n_D=n_C$.}, equivalently, $[C:D]>0$, and we write $D\subset C$ if $D\subseteq C$ but $D\neq C$. Also, for multi-complexes $D\subset C$, we now let $\langle D \rangle$ be the smallest sub-multi-complex of $C$ such that $n_{\langle D \rangle}=n_C$ (this differs slightly from before). One can easily see that if $D \subset C$, then $\langle D \rangle \subset C$. Finally, we note that with respect to the grading on $H_{\mathcal{C}}$ by the number of elements in the set on which a multi-complex is based, the elements $P_C$ are homogeneous of degree $|n_C|$.

	\begin{mydef}\label{definition: integral basis}
		Let $q=\{q_C\}_{C \in \mathcal{T}}$ be a basis of the space of primitives of $H_\mathcal{C}$, considered as an algebra over a field of characteristic zero. We say that $q=\{q_C\}_{C \in \mathcal{T}}$ is an \emph{integral basis} if for each multi-complex $C=D_1\cdots D_t$ with connected components $D_i$, and $q_C:=q_{D_1}\cdots q_{D_t}$, the following two conditions hold:
		\begin{enumerate}
			\item 
			Each $C \in \mathcal{T}$ can be expressed as $C=f_C(q_D\mid D \subseteq C)$, where $f_C$ is a polynomial with non-negative integers. 
			\item 
			$\mathbb{Z}[q_C]_{C \in \mathcal{T}}$ is a subset of $\mathbb{Z}[C]_{C \in \mathcal{T}}$, that is, each $q_C$ is a polynomial in $\{D\}_{D\in \mathcal{T}}$, with integral coefficients.
		\end{enumerate}
	\end{mydef}
	
	\begin{pro}
		Let $q=\{q_C\}_{C \in \mathcal{T}}$ be an integral basis. Then, for each multi-complex $C$, there exist integers $\alpha_{C,D}$ such that 
		\begin{equation}\label{eq: integral basis}
		C=\sum_{\substack{[D]:\\ D \subseteq C}}\alpha_{C,D}q_D,
		\end{equation}
where the sum runs over isomorphism classes $[D]$ of $D \subseteq C$. 
	\end{pro}
	\begin{proof}
	    Note that the condition (1) of an integral basis for $C\in \mathcal{T}$ implies the same condition holds for all $C\in\mathcal{C}$. Also, by the formulas relating $C$ and $P_C$, it is enough to prove that $P_C\in {\rm Span}_{\mathbb{Z}}\{q_D \mid D\subseteq C\}$. The same formulas relating the $P_C$'s with the $C$'s we observe we have a formula of the type 
	    $$P_C=g_C(q_D | D\subseteq C)$$
        for each $C$. Regarding $H_\mathcal{C}$ as a polynomial Hopf algebra in the $q_C$ for $C\in\mathcal{T}$, since $P_C$ is a primitive element, it follows that $g_C$ is a linear polynomial (with integer coefficients). This proves the clam. 

	\end{proof}

\begin{mydef}\label{definition: combinatorial integral basis}
		An integral basis $q=\{q_C\}_{C \in \mathcal{T}}$ is said to be \emph{combinatorial} if for each multi-complex $C$ and each sub-multi-complex $D\preceq C$, there exist positive integers $\gamma_{C,D}$,  such that
		\begin{equation}\label{eq: combinatorial eqeq}
		C=\sum_{ D \preceq C}\gamma_{C,D}q_{D},
		\end{equation}
	\end{mydef}

\begin{rmk}
Among integral bases, there are some  that are ``combinatorial'' in the sense that the formula \eqref{eq: integral basis} can be written as in \eqref{eq: combinatorial eqeq}. By gathering similar terms in this last formula, this is equivalent to saying that all the $\alpha_{C,D}$ coefficients from $C= \sum\limits_{\substack{[D]:\\ D \subseteq C}} \alpha_{C,D}q_D$ have the property that eiher $\alpha_{C,D}=0$ if $\deg(D)<\deg(C)$ (they have different vertex sets) or  satisfy the following inequality otherwise:
\begin{equation}\label{eq: combinatorial formula}
\alpha_{C,D} \geq [C:D]. 
\end{equation}
For instance, the basis $\{P_C\}_{C\in \mathcal{T}}$ trivially satisfies this condition. In fact, this motivated us to use the term ``combinatorial'' since in the case of $\{P_C\}_{C\in \mathcal{T}}$, the coefficient $\alpha_{C,D}$ is the number of sub-multi-complexes of $C$ based on the same vertex set and which are isomorphic to $D$. Hence, an integral basis $\{q_C\}_{C \in \mathcal{T}}$ is combinatorial if there are coefficients $\alpha_{C,D}$ for $D\preceq C$, which are determined up to the isomorphism classes of $C,D$, such that $\alpha_{C,D}\geq [C:D]$ and
$$C=\sum_{\substack{[D]:\\ D \preceq C}}\alpha_{C,D}q_D,$$
where this last sum takes place over isomorphism classes of complexes $[D]$.
\end{rmk}
	
Let $X$ be the set of combinatorial integral bases of the space of primitive elements of $H_\mathcal{C}$. For two elements $q=\{q_C\}$ and $r=\{r_C\}$ of $X$, we denote $q \leq r$ if the coefficients $\alpha_{C,D}$ and $\beta_{C,D}$ in the formulas 	
\[
C=\sum_{\substack{[D]:\\ D \preceq C}}\alpha_{C,D}q_{D}=\sum_{\substack{[D]:\\ D \preceq C}}\beta_{C,D}r_{D},
\]
have the property that $\alpha_{C,D} \leq \beta_{C,D}$ for all $D \preceq C$. One can easily see that this relation defines a partial order on the set $X$. 
Then, clearly the basis $\{P_C\}_{C\in \mathcal{T}}$ is the minimal element in $X$ because of \eqref{eq: combinatorial formula}. Hence, we have the following. 

\begin{pro}
The basis $\{P_C\}_{C\in \mathcal{T}}$ is the minimal element in the set $X$ of combinatorial integral bases of the space of primitive elements of $H_\mathcal{C}$ with the aforementioned partial order.
\end{pro}

Furthermore $\{P_C\}_{C\in \mathcal{T}}$ also satisfies a certain universal property as follows. 

\begin{pro}
Let $\{q_C\}$ be a combinatorial integral basis. Then, $\{P_C\}_{C\in \mathcal{T}}$ can be written as a linear combination of $\{q_C\}$ with non-negative coefficients. 
\end{pro}
\begin{proof}
We prove the following equation by induction on the number $e=|C|$. 
\begin{equation}\label{eq: induction}
P_C = q_C + \sum_{\substack{[D]:\\ D \preceq C}} b_{C,D}q_D, \quad \textrm{for some } b_{C,D} \in \mathbb{N}. 
\end{equation}
When $e=1$, it is clear, since the one point multi-complex $C$ is primitive and we must have $q_C=P_C$. Now, for the induction step, consider the following
\begin{equation}\label{eq: induction step}
C=P_C + \sum_{\substack{[D]:\\ D \prec C}} [C:D] P_D.
\end{equation}
We can replace each $P_D$ with $\{q_C\}$ from \eqref{eq: induction}. In particular, we obtain 
\begin{equation}
C=q_C + \sum_{\substack{[D]:\\ D \preceq C}} c_{C,D} q_D, 
\end{equation}
for some integers $c_{C,D}$. We note that $c_{C,D} \geq [C:D]$ since the formula \eqref{eq: induction} has the term $q_C$. Therefore, we finally obtain 
\[
q_C + \sum_{\substack{[D]:\\ D \subset C}}c_{C,D}P_D = C = P_C + \sum_{\substack{[D]:\\ D \prec C}} [C:D] P_D,
\]
showing that 
\[
P_C = q_C + \sum_{\substack{[D]:\\ D \subset C}} (c_{C,D} -[C:D])P_D. 
\]
Since $(c_{C,D} -[C:D])$ are non-negative integers, this proves our proposition. 
\end{proof}
	
	\subsection*{The reconstruction conjectures}
	
	Let $G$ be a graph. The vertex deck of the $G$ is defined to be the multiset of (isomorphism types of) graphs obtained by removing some subset $X\subseteq V(G)$ and taking the induced graph of $G$ graph $V(G)-X$. The edge deck of $G$ is defined to be the multiset of (isomorphism type of) graphs $(G-E)_{E\subseteq E(G)}$. The following two conjectures are well known in combinatorics:
	
	\vspace{.3cm}
	
	\noindent
	{\bf Conjecture} [Vertex Reconstruction Conjecture]
	{\it If two graphs have the same vertex decks, then they are isomorphic.}
	
	\vspace{.3cm}
	
	\noindent
	{\bf Conjecture} [Edge Reconstruction Conjecture]
	{\it If two graphs have the same edge decks, then they are isomorphic.}
	
	\vspace{.3cm}
	
	The vertex reconstruction conjecture is known to hold for some classes of graphs, such as trees. It is also well-known that the vertex reconstruction conjecture implies the edge reconstruction conjecture (see, for instance, \cite{hemminger1969reconstructing}).   We note now that this also follows as a consequence of our algebraic setup:
	
	\begin{cor}
		The vertex reconstruction conjecture implies the edge reconstruction conjecture. 
	\end{cor}
	\begin{proof}
		Let $G,G'$ be two graphs. If $G$ and $G'$ have the same edge deck, then the cancellation-free expression of a graph in \eqref{eq 1} and \eqref{eq 2} (Remark \ref{remark: cancellation-free graphs}) shows that $G- G' = P_G - P_{G'}$. The same formula now implies that $G$ and $G'$ have the same vertex deck. Indeed, if not, there exists a graph $H=G-\{v\}$ for some vertex $v$ of $G$ such that the multiplicities of $H$ in the vertex decks of $G$ and $G'$ are different; but these multiplicities are, in this case, equal also to $[G:L]$ and $[G':L]$, respectively, where $L=H\sqcup \{v\}$. Then, using equation $\eqref{eq 2}$ again, we see that $G-G'$ we would have a non0-zero term $([G:L]-[G':L])L$, which has degree $\geq 2$, and thus $G-G'$ would not be primitive (a graph with at least two vertices cannot be a primitive element).
	\end{proof}
	
	We can also prove the following case of the vertex reconstruction conjecture by using the Hopf algebras of graphs. 
	
	\begin{pro}
		The vertex reconstruction conjecture is true for disconnected graphs. 	
	\end{pro}
	\begin{proof}
		Let $G=G_1\cdots G_k$, where $G_i$ are connected components of $G$. Suppose that $H=H_1\cdots H_t$ is a graph with the same vertex deck as $G$. Let $G_i=f_i(P)$, a polynomial of the  primitives $P=\{P_L\}_L$ as in Theorem \ref{theorem: multi-complex of dimension 1}. Similarly, we write $H_i=g_i(P)$. As $G$ and $H$ have the same vertex deck, we know that $G-H$ is a primitive element; but $G-H=f_1(P)\cdots f_k(P)-g_1(P)\cdots g_t(P)$, and as $G$ and $H$ are disconnected, neither of $f_1(P)\cdots f_k(P)$ or $g_1(P)\cdots g_t(P)$ can have a monomial of degree $1$ (nor a constant term), showing that $G-H=0$, and hence $G=H$. 
	\end{proof}

	\bibliography{hopf}\bibliographystyle{alpha}

\end{document}